\documentclass[reqno]{amsart}
\usepackage[colorlinks,linkcolor=black,anchorcolor=black,citecolor=black,hyperindex=black,CJKbookmarks=black]{hyperref}

\newtheorem{theorem}{Theorem}[section]
\newtheorem{lemma}[theorem]{Lemma}

\theoremstyle{definition}

\newtheorem{remark}[theorem]{Remark}

\numberwithin{equation}{section}

\begin{document}

\title[The denominators of convergents for continued fractions]{The denominators of convergents for continued fractions}

\author {Lulu Fang, Min Wu and Bing Li$^{*}$}
\address{School of Mathematics, South China University of Technology, Guangzhou 510640, P.R. China}
\email{f.lulu@mail.scut.edu.cn, wumin@scut.edu.cn and scbingli@scut.edu.cn}


\thanks {* Corresponding author}

\begin{abstract}
For any real number $x \in [0,1)$, we denote by $q_n(x)$ the denominator of the $n$-th convergent of the continued fraction expansion of $x$ $(n \in \mathbb{N})$.
It is well-known that the Lebesgue measure of the set of points $x \in [0,1)$ for which $\log q_n(x)/n$ deviates away from $\pi^2/(12\log2)$ decays to zero as $n$ tends to infinity.
In this paper, we study the rate of this decay by giving an upper bound and a lower bound. What is interesting is that the upper bound is closely related to the Hausdorff dimensions of the level sets for $\log q_n(x)/n$.
As a consequence, we obtain a large deviation type result for $\log q_n(x)/n$, which indicates that the rate of this decay is exponential.
\end{abstract}

\maketitle

\section{Introduction}
Let $T: [0,1) \longrightarrow [0,1)$ be the \emph{continued fraction transformation} defined as
\begin{equation*}
T(0):=0\ \ \  \text{and}\ \ \ T(x):= 1/x- [1/x]\ \ \  \text{if}\ \ \  x \in (0,1).
\end{equation*}
where $[x]$ denotes the greatest integer not exceeding $x$.
Then every real number $x \in [0,1)$ can be uniquely written as
\begin{equation}\label{continued fraction expansion}
x = \dfrac{1}{a_1(x) +\dfrac{1}{a_2(x) + \ddots +\dfrac{1}{a_n(x)+ \ddots}}},
\end{equation}
where $a_1(x) = [1/x]$ and $a_{n +1}(x) = a_1(T^nx)$ for all $n \geq 1$. The representation (\ref{continued fraction expansion}) is said to be the \emph{continued fraction expansion} of $x$ and $a_n(x), n \geq 1$  are called the \emph{partial quotients} of the continued fraction expansion of $x$.  Sometimes we write the form (\ref{continued fraction expansion}) as $[a_1(x), a_2(x), \cdots, a_n(x), \cdots]$. For any $n \geq 1$, we denote by $\frac{p_n(x)}{q_n(x)}:= [a_1(x), a_2(x), \cdots, a_n(x)]$ the $n$-th \emph{convergent} of the continued fraction expansion of $x$, where $p_n(x)$ and $q_n(x)$ are relatively prime. With the conventions $p_{-1}=1$, $q_{-1}=0$, $p_0=0$, $q_0=1$, the quantities $p_n$ and $q_n$ satisfy the following recursive formula:
\begin{equation}\label{recursive}
p_n(x) = a_n(x) p_{n-1}(x) + p_{n-2}(x)\ \ \ \text{and}\ \ \ q_n(x) = a_n(x) q_{n-1}(x) + q_{n-2}(x).
\end{equation}
It is easy to see that these convergents are rational numbers and $p_n(x)/q_n(x) \rightarrow x$ as $n \rightarrow \infty$ for all $x \in [0,1)$. More precisely,
\begin{equation}\label{diophantine}
\frac{1}{2q_{n+1}^2(x)} \leq \frac{1}{2q_n(x)q_{n+1}(x)} \leq \left|x-\frac{p_n(x)}{q_n(x)}\right| \leq \frac{1}{q_n(x)q_{n+1}(x)} \leq \frac{1}{q_n^2(x)}.
\end{equation}
This is to say that the speed of $p_n(x)/q_n(x)$ approximating to $x$ is dominated by $q_n^{-2}(x)$. So the denominator of the $n$-th convergent $q_n(x)$ plays an important role in the problem of Diophantine approximation. For more details about continued fractions, we refer the reader to two monographs of Iosifescu and Kraaikamp \cite{lesIK02} and Khintchine \cite{lesKhi64}.

For an irrational number $x \in [0,1)$, we denote
\[
\beta_*(x) = \liminf\limits_{n \to \infty}\frac{\log q_n(x)}{n}\ \ \ \ \text{and}\ \ \ \
\beta^*(x) = \limsup\limits_{n \to \infty}\frac{\log q_n(x)}{n}.
\]
The functions $\beta_*(x)$ and $\beta^*(x)$ are called the \emph{lower} and \emph{upper L\'{e}vy constant} of $x$ respectively. If $\beta_*(x) = \beta^*(x)$, we say that $x$ has a L\'{e}vy constant and denote the common value by $\beta(x)$. It is not difficult to check that $\beta_*(x) \geq \gamma_0:=\log((\sqrt{5}+1)/2)$ for all irrational number $x$. On the one hand, Faivre \cite{lesFai92} showed that every quadratic irrational has a L\'{e}vy constant. In 2006, Wu \cite{lesWu06} proved that the set of L\'{e}vy constants of quadratic irrationals is dense in the interval $[\gamma_0,+\infty)$.
On the other hand, Faivre \cite{lesFai97} showed that for any $\gamma \geq \gamma_0$, there exists an irrational number $x$ such that $x$ has L\'{e}vy constant $\lambda$. Recently, Baxa \cite{lesBaxa09} improved this result for transcendental numbers. That is to say, there exists a transcendental number $x$ such that $\beta(x)= \gamma$ for any $\gamma \geq \gamma_0$.
Also, Baxa \cite{lesBaxa99} obtained that for any two real numbers satisfying $\gamma_0\leq \gamma_1 \leq \gamma_2 < +\infty$, there exist non-denumerably many pairwise not equivalent irrational numbers $x$ such that $\beta_*(x) = \gamma_1$ and $\beta^*(x) = \gamma_2$. Furthermore, Wu \cite{lesWu06a} considered the Hausdorff dimension of the set of such points and gave it a lower bound.
A basic result about L\'{e}vy constant is due to L\'{e}vy \cite{lesLevy29}, who proved that the function $\beta(x)$ is constantly $\pi^2/(12\log2)$ for $\lambda$-almost all $x \in [0,1)$. Here $\lambda$ denotes the Lebesgue measure on $[0,1)$.

\begin{theorem}[\cite{lesLevy29}]\label{Levy theorem}
For $\lambda$-almost all $x \in [0,1)$,
\begin{equation*}
\lim\limits_{n \to \infty}\frac{\log q_n(x)}{n} = \frac{\pi^2}{12\log2}.
\end{equation*}
\end{theorem}

From the fractal dimension points of view, Barreira and Schmeling \cite{lesBS00} pointed out that the set of points $x \in [0,1)$ for which the limit in Theorem \ref{Levy theorem} does not exist (i.e., $\beta_*(x)<\beta^*(x)$) has full Hausdorff dimension.
Furthermore, Pollicott and Weiss \cite{lesPW99} first considered the multifractal analysis of $\beta(x)$ and proved that the spectral function
\[
\tau(\gamma):= \dim_{\rm H} \left\{x \in [0,1): \beta(x) = \gamma \right\} = \frac{\inf_{\theta \in \mathbb{R}}\{\theta\cdot 2\gamma + \mathrm{P}(\theta)\}}{2\gamma}
\]
for any $\gamma \geq \gamma_0$ (see also Fan et al.~\cite{lesF.L.W.W} and Kesseb\"{o}hmer and Stratmann \cite{lesKS07}),
where $\dim_{\rm H}$ denotes the Hausdorff dimension and $\mathrm{P}(\cdot)$ is called the \emph{Diophantine pressure function} given by
\[
\mathrm{P}(\theta) = \lim_{n \to \infty} \frac{1}{n} \log \sum_{a_1,\cdots,a_n} q_n^{-2\theta}([a_1,\cdots,a_n])\ \ \text{for any}\  \theta > 1/2.
\]
It is worth remarking that the spectral function $\tau(\cdot)$ is real-analytic on the interval $(\gamma_0,+\infty)$ satisfying $\tau(\gamma)$ goes to $1/2$ as $\gamma$ tens to infinity, it is increasing on the interval $[\gamma_0, \pi^2/(12\log2)]$ and decreasing on the interval $[\pi^2/(12\log2), +\infty)$, and it also has a unique maximum value equal to 1 at point $\pi^2/(12\log2)$; the Diophantine pressure function $\mathrm{P}(\cdot)$ has a singularity at $1/2$ and is decreasing, convex and real-analytic on $(1/2,+\infty)$ satisfying
\begin{equation}\label{ele}
\mathrm{P}(1)=0\ \ \ \ \ \ \ \ \ \text{and}\ \ \ \ \ \ \ \ \ \mathrm{P}^\prime(1)=-\pi^2/(6\log 2).
\end{equation}
More detailed analyses of $\tau(\cdot)$ and $\mathrm{P}(\cdot)$ can be founded in Fan et al.~\cite{lesF.L.W.W}, Kesseb\"{o}hmer and Stratmann \cite{lesKS07}, Mayer \cite{lesMay90} and Pollicott and Weiss \cite{lesPW99}. From the metrical points of view, some limit theorems about $q_n(x)$ have been extensively investigated.
For instance, Ibragimov \cite{lesIbr61} proved that the distribution of the $\log q_n(x)$, suitably normalized, converges to the normal distribution with mean 0 and unit variance, that is, the central limit theorem for $q_n(x)$. Furthermore, Morita \cite{lesMor94} showed that the Berry-Esseen bound for the above central limit theorem is as we would expect $\mathcal{O}(n^{-1/2})$. Later, Philipp and Stackelberg \cite{lesPS69} obtained the classical law of the iterated logarithm for $q_n(x)$ (see also Gordin and Reznik \cite{lesGR70}).

It is worth noting that these classical limit theorems basically concern that the averages taken over large samples converge to expectation values in some sense, but say little or nothing about the rate of convergence.
It follows from Theorem \ref{Levy theorem} that the Lebesgue measure of the set of points $x$ for which $\log q_n(x)/n$ deviates away from $\pi^2/(12\log2)$ decays to zero as $n$ goes to infinity.
A natural question is arisen: what are the rates of these decreasing probabilities?
In fact,
Fang et al.~\cite{lesFWSL15} have considered these decays and showed that the upper bounds of these decays are exponential.
In this paper, we not only obtain the upper and lower bounds of these decreasing probabilities, but also give them explicit formulae. And an interesting phenomenon is that the explicit formula of the upper bound is closely related to the spectral function $\tau(\cdot)$ (see Remarks \ref{1} and \ref{2} below).

\section{Main results}
In this section, we will state our main results. For simplicity, we use the notation $b$ to denote the constant $\pi^2/(12\log2)$.

\begin{theorem}\label{dayu}
For any $\varepsilon >0$, we have
\begin{equation*}\
\limsup_{n \to \infty} \frac{1}{n}\log \lambda\left\{x\in [0,1):\frac{\log q_n(x)}{n} \geq b + \varepsilon\right\} \leq \theta_1(\varepsilon)
\end{equation*}
and
\begin{equation*}
\liminf_{n \to \infty} \frac{1}{n}\log \lambda\left\{x\in [0,1):\frac{\log q_n(x)}{n} \geq b + \varepsilon\right\} \geq -2\log b_\varepsilon - \log 3
\end{equation*}
where $\theta_1(\varepsilon) = \inf\limits_{0< t < 1} \big\{-t(b + \varepsilon) + \mathrm{P}(1-t/2)\big\}< 0$ and $b_\varepsilon$ denotes the smallest integer no less than $e^{b+\varepsilon}$.
\end{theorem}

\begin{remark}\label{1}
By the domain of the function $\mathrm{P (\cdot)}$, we can write $\theta_1(\varepsilon)$ as
\[
\theta_1(\varepsilon) = \inf\limits_{t < 1} \big\{-t(b + \varepsilon) + \mathrm{P}(1-t/2)\big\}.
\]
In fact, for any $\varepsilon >0$, let $f(t) = -t(b + \varepsilon) + \mathrm{P}(1-t/2)$ for any $t \leq 0$. Since $\mathrm{P(\cdot)}$ is convex and real-analytic on $(1/2,+\infty)$, we know $\mathrm{P}^\prime(1-t/2) \geq \mathrm{P}^\prime(1)$ for any $t \leq 0$. It follows from (\ref{ele}) that
\[
f^\prime(t) = -(b+\varepsilon) -2^{-1}\cdot \mathrm{P}^\prime(1-t/2) \leq -(b+\varepsilon) -2^{-1}\cdot \mathrm{P}^\prime(1) = -\varepsilon <0
\]
for any $t \leq 0$. So $f(\cdot)$ is non-increasing on $(-\infty,0]$ and hence $f(t) \geq 0$ for any $t \leq 0$. As a consequence, it is easy to check that
\[
\theta_1(\varepsilon) = 2(b+\varepsilon)\big(\tau(b+\varepsilon)-1\big)
\]
and hence that $-(b+\varepsilon) < \theta_1(\varepsilon)<0$ and $\theta_1(\varepsilon)$ tends to zero as $\varepsilon$ goes to zero since the spectral function $\tau(\cdot)$ has a unique maximum value equal to 1 at point $b$.
\end{remark}

\begin{theorem}\label{xiaoyu}
For any $0< \varepsilon \leq b$, we have
\begin{equation*}
\limsup_{n \to \infty} \frac{1}{n}\log \lambda\left\{x\in [0,1):\frac{\log q_n(x)}{n} \leq b -\varepsilon\right\} \leq \theta_2(\varepsilon)
\end{equation*}
and for any $0< \varepsilon \leq b - \log 2$,
\begin{equation*}
\liminf_{n \to \infty} \frac{1}{n}\log \lambda\left\{x\in [0,1):\frac{\log q_n(x)}{n} \leq b -\varepsilon\right\} \geq -2\log b^{\ast}_\varepsilon - \log 3,
\end{equation*}
where $\theta_2(\varepsilon) = \inf\limits_{t> 0} \big\{t(b - \varepsilon) + \mathrm{P}(1+t/2)\big\}< 0$ and $b^{\ast}_\varepsilon$ denotes the largest integer no greater than $(e^{b-\varepsilon}-1)$.
\end{theorem}

\begin{remark}\label{2}
Being similar to Remark \ref{1}, $\theta_2(\varepsilon)$ can be written as
\[
\theta_2(\varepsilon) = \inf\limits_{t >- 1} \big\{t(b - \varepsilon) + \mathrm{P}(1+t/2)\big\}
\]
for any $0 < \varepsilon \leq b$. Moreover, it also has another alternative form
\[
\theta_2(\varepsilon) = 2(b-\varepsilon)\big(\tau(b-\varepsilon)-1\big),
\]
which only holds for $0 <\varepsilon \leq b - \log((\sqrt{5}+1)/2)$ form the definition of $\tau(\cdot)$. In this case, it is clear to see that $-2(b-\varepsilon) \leq \theta_2(\varepsilon) <0$ and $\theta_2(\varepsilon)$ tends to zero as $\varepsilon$ goes to zero.
\end{remark}

The following is a result of large deviations for $\log q_n(x)/n$, which improves the result of Theorem \ref{Levy theorem} by Borel-Cantelli lemma.

\begin{theorem}\label{Large deviations}
For any $\varepsilon > 0$, there exist constants $A,B > 0$ and $\alpha, \beta > 0$ (both only depending on $\varepsilon$) such that for all $n \geq 1$, we have
 \[
Be^{-\beta n} \leq \lambda\left\{x \in [0,1):\left|\frac{\log q_n(x)}{n}- \frac{\pi^2}{12\log2}\right| \geq \varepsilon\right\} \leq Ae^{-\alpha n}.
\]
\end{theorem}

\section{The proofs of theorems}

This section is devoted to giving the proofs of our main results. We denote by $\mathbb{I}$ the set of all irrational numbers in $[0,1)$ and use the notation $\mathrm{E}(\xi)$ to denote the expectation of a random variable $\xi$ w.r.t.~the Lebesgue measure $\lambda$.
 For any $n \in \mathbb{N}$ and $a_1, a_2, \cdots, a_n\in \mathbb{N}$, we call
\begin{equation*}
I(a_1, \cdots, a_n):= \left\{x \in \mathbb{I}: a_1(x)=a_1, \cdots, a_n(x)=a_n\right\}
\end{equation*}
the $n$-th order \emph{cylinder} of continued fractions. It is well-known (see \cite{lesDK02, lesIK02}) that $I(a_1, \cdots, a_n)$ is an interval with the endpoints $p_nq_n^{-1}$ and $(p_n +p_{n-1})(q_n +q_{n-1})^{-1}$. As a consequence, the length of $I(a_1, \cdots, a_n)$ denoted by $|I(a_1, \cdots, a_n)|$, is equal to $q_n^{-1}(q_n+q_{n-1})^{-1}$, where the quantities $p_n$ and $q_n$ are obtained by the recursive formula (\ref{recursive}).
The following lemma establishes a relation between the Diophantine pressure function $\mathrm{P}(\cdot)$ and the growth of the expectation of $q_n$, which plays an important role in our proofs.

\begin{lemma}\label{proofs}
For any $\theta <1/2$,
\[
\mathrm{P}(1-\theta) = \lim_{n \to \infty}\frac{1}{n} \log \mathrm{E}\left(q_n^{2\theta}\right).
\]
\end{lemma}

\begin{proof}
By the definition of expectation, we know that
\begin{equation}\label{qiwang}
E\left(q_n^{2\theta}\right) = \sum_{a_1,\cdots,a_n} q_n^{2\theta}([a_1,\cdots,a_n])\cdot \lambda\left(I(a_1,\cdots,a_n)\right),
\end{equation}
where $a_1,\cdots,a_n$ run over all the positive integers.
Since
\begin{equation*}
\frac{1}{2q_n^2([a_1,\cdots,a_n])}\leq \lambda\left(I(a_1,\cdots,a_n)\right) = |I(a_1, \cdots, a_n)| \leq \frac{1}{q_n^2([a_1,\cdots,a_n])},
\end{equation*}
combing this with (\ref{qiwang}), we deduce that
\[
\frac{1}{2}\cdot \sum_{a_1,\cdots,a_n} q_n^{-2(1-\theta)}([a_1,\cdots,a_n]) \leq E\left(q_n^\theta\right) \leq \sum_{a_1,\cdots,a_n} q_n^{-2(1-\theta)}([a_1,\cdots,a_n])
\]
and hence that
\[
\mathrm{P}(1-\theta) = \lim_{n \to \infty}\frac{1}{n} \log \mathrm{E}\left(q_n^{2\theta}\right) \ \ \ \text{for any $\theta <1/2$}.
\]
\end{proof}

\subsection{Proof of Theorem \ref{dayu}}
The proof is divided into two parts: limsup part and liminf part. The proof of limsup part heavily relies on the Markov's inequality and Lemma \ref{proofs}. The idea of the proof of liminf part is from finding a subset inside whose Lebesgue measure decays to 0 exponentially.

\begin{proof}[Proof of the limsup part]
Let $0<t<1$ be a parameter. Notice that
\begin{equation*}
\lambda \left\{x \in \mathbb{I}:\frac{\log q_n(x)}{n} \geq b + \varepsilon\right\} = \lambda \left\{x \in \mathbb{I}: q_n^{t}(x) \geq e^{nt(b + \varepsilon)}\right\},
\end{equation*}
in view of Markov's inequality, we have that
\begin{equation}\label{chebyshev}
\lambda \left\{x \in \mathbb{I}: \frac{\log q_n(x)}{n} \geq b + \varepsilon\right\} \leq e^{-nt(b+ \varepsilon)}\cdot E\left(q_n^{t}\right).
\end{equation}

By Lemma \ref{proofs}, we know
\[
\lim_{n \to \infty}\frac{1}{n} \log \mathrm{E}\left(q_n^{t}\right)= \mathrm{P}(1-t/2).
\]
Hence, for any $\eta >0$, there exists a positive number $N$ (depending on $\eta$) such that for all $n \geq N$, we have
\[
\mathrm{E}\left(q_n^{t}\right) \leq e^{n(\mathrm{P}(1-t/2) + \eta)}.
\]
Fixed such $n \geq N$, it follows from (\ref{chebyshev}) that
\begin{equation}\label{huajian}
\lambda \left\{x \in \mathbb{I}: \frac{\log q_n(x)}{n} \geq b + \varepsilon\right\} \leq e^{-nt(b+ \varepsilon)+n(\mathrm{P}(1-t/2) + \eta)}.
\end{equation}
Taking the logarithm on both sides of the inequality (\ref{huajian}), we deduce that
\begin{equation*}
\frac{1}{n}\log \lambda \left\{x \in \mathbb{I}: \frac{\log q_n(x)}{n} \geq b + \varepsilon\right\} \leq - t(b + \varepsilon) + \mathrm{P}(1-t/2) + \eta.
\end{equation*}
Thus, for all $0<t<1$, we obtain that
\begin{equation*}
\limsup_{n \to \infty} \frac{1}{n}\log \lambda \left\{x \in \mathbb{I}: \frac{\log q_n(x)}{n} \geq b + \varepsilon\right\} \leq -t(b + \varepsilon) + \mathrm{P}(1-t/2)
\end{equation*}
since $\eta >0$ is arbitrary.
Therefore,
\begin{equation*}
\limsup_{n \to \infty} \frac{1}{n}\log \lambda \left\{x \in \mathbb{I}: \frac{\log q_n(x)}{n} \geq b + \varepsilon\right\} \leq \theta_1(\varepsilon)
\end{equation*}
with
\[
\theta_1(\varepsilon) = \inf_{0< t < 1} \big\{-t(b + \varepsilon) + \mathrm{P}(1-t/2)\big \}.
\]
Now it remains to show that $\theta_1(\varepsilon) < 0$. In fact, let $h(u)$ be the function defined as
\[
h(u) = -u(b + \varepsilon) + \mathrm{P}(1-u/2)\ \ \text{for any}\ u <1 .
\]
In view of (\ref{ele}), it is easy to check that $h(0) = 0$ and $h^{\prime}(0) = -\varepsilon < 0$. Hence there exists $u_0 >0$ such that $h(u_0) < 0$ by the definition of derivative. Therefore, we complete the proof of the limsup part by observing that $\theta_1(\varepsilon) \leq h(u_0) < 0$.
\end{proof}

To prove the liminf part, we need the following lemma (see \cite{lesKhi64}).

\begin{lemma} [\cite{lesKhi64}]\label{Khin57}
For any $n \geq 1$ and $a_1, \cdots, a_n, a_{n+1}\in \mathbb{N}$, we have
\begin{equation*}
\frac{1}{3a_{n+1}^2}\leq \frac{|I(a_1, \cdots, a_n, a_{n+1})|}{|I(a_1, \cdots, a_n)|} \leq \frac{2}{a_{n+1}^2}.
\end{equation*}
\end{lemma}

\begin{proof}
For any $n \geq 1$ and $a_1, \cdots, a_n, a_{n+1}\in \mathbb{N}$,
we know that
\[
|I(a_1, \cdots, a_n)| = \frac{1}{q_n(q_n+q_{n-1})}\ \ \text{and}\ \ |I(a_1, \cdots, a_n, a_{n+1})| = \frac{1}{q_{n+1}(q_{n+1}+q_n)},
\]
where the quantities $p_{n-1}$, $q_{n-1}$, $p_n$, $q_n$, $p_{n+1}$ and $q_{n+1}$ satisfy the recursive formula (\ref{recursive}). Therefore,
\begin{equation}\label{Khinchine57}
\frac{|I(a_1, \cdots, a_n, a_{n+1})|}{|I(a_1, \cdots, a_n)|}  = \frac{q_n(q_n+q_{n-1})}{q_{n+1}(q_{n+1}+q_n)} = \frac{1}{a^2_{n+1}} \cdot \frac{1+y_n}{(1+z_n)(1+1/a_{n+1}+z_n)},
\end{equation}
where $y_n = q_{n-1}/q_n \in [0,1)$, $z_n = y_n/a_{n+1}$ and the last equation follows from the recursive formula $q_{n+1} = a_{n+1}q_n + q_{n-1}$. The second factor on the last term of (\ref{Khinchine57}) is obviously not greater than 2 since $y_n<1$, $a_{n+1} >0$ and $z_n >0$. Notice that $a_{n+1} \geq 1$ and $0 \leq y_n<1$, we deduce that
\[
\frac{1+y_n}{1+z_n} \geq 1\ \ \text{and}\ \ 1+\frac{1}{a_{n+1}}+z_n \leq 3.
\]
This implies that the second factor on the last term of (\ref{Khinchine57}) is not less than 1/3. Thus, we complete the proof.
\end{proof}

We are ready to give the proof of the part of liminf in Theorem \ref{dayu}.

\begin{proof}[Proof of the liminf part]
For any $x \in [0,1)$, by the recursive formula (\ref{recursive}), we know that
\begin{equation*}
q_n(x)= a_n(x) q_{n-1}(x) + q_{n-2}(x) \geq a_n(x) q_{n-1}(x)\geq \cdots \geq a_n(x)\cdots a_1(x).
\end{equation*}
Hence that
\begin{align}\label{xiao}
\left\{x \in \mathbb{I}: \frac{\log q_n(x)}{n} \geq b + \varepsilon\right\} \supseteq
\left\{x \in \mathbb{I}:\frac{\log a_1(x) \cdots a_n(x)}{n} \geq b + \varepsilon\right\}.
\end{align}
Let $b_\varepsilon$ be the smallest integer no less than $e^{b+\varepsilon}$.
Since
\begin{align*}
\left\{x \in \mathbb{I}:\frac{\log a_1(x) \cdots a_n(x)}{n} \geq b + \varepsilon\right\} \supseteq
 \Big\{x \in \mathbb{I}: \log a_1(x) \geq b + \varepsilon, \cdots, \log a_n(x) \geq b + \varepsilon\Big\}
\end{align*}
and
\begin{align*}
\Big\{x \in \mathbb{I}: \log a_1(x) \geq b + \varepsilon, \cdots, \log a_n(x) \geq b + \varepsilon\Big\} \supseteq
\Big\{x \in \mathbb{I}:a_1(x) = b_\varepsilon, \cdots, a_n(x)= b _\varepsilon\Big\},
\end{align*}
combing these with (\ref{xiao}), we deduce that
\begin{align*}
\lambda\left\{x \in \mathbb{I}:\frac{\log q_n(x)}{n} \geq b+ \varepsilon\right\}
&\geq  \lambda \big\{x \in \mathbb{I}:a_1(x) = b_\varepsilon, \cdots, a_n(x)= b _\varepsilon\big\} \\
&= |I(\underbrace{b_\varepsilon, \cdots,b_\varepsilon}_{n})| \\
& \geq \frac{1}{3b_\varepsilon^2} \cdot |I(\underbrace{b_\varepsilon, \cdots,b_\varepsilon}_{n-1})|,
\end{align*}
where the last inequality follows from Lemma \ref{Khin57}. Repeating this procedure $(n-1)$ times, we obtain that
\begin{align*}
\lambda\left\{x \in \mathbb{I}:\frac{\log q_n(x)}{n}  \geq b+ \varepsilon\right\} &\geq \left(\frac{1}{3b_\varepsilon^2}\right)^{n-1} \cdot |I(b_\varepsilon)| \\
&= \left(\frac{1}{3b_\varepsilon^2}\right)^{n-1} \cdot \frac{1}{b_\varepsilon(b_\varepsilon+1)} \geq \left(\frac{1}{3b_\varepsilon^2}\right)^{n}.
\end{align*}
Therefore,
\[
\liminf_{n \to \infty} \frac{1}{n}\log \lambda \left\{x \in \mathbb{I}:\frac{\log q_n(x)}{n} \geq  b+ \varepsilon\right\} \geq - 2\log b_\varepsilon -\log 3.
\]
This gives a lower bound of the desired result. Next we will point out that $- 2\log b_\varepsilon - \log 3 \leq \theta_1(\varepsilon)$. By the definition of $b_\varepsilon$, we know that $- 2\log b_\varepsilon \leq -2(b+\varepsilon)$.
It follows from Remark \ref{1} that $-(b+\varepsilon) \leq \theta_1(\varepsilon)$. So $- 2\log b_\varepsilon-\log 3 < \theta_1(\varepsilon)$.
\end{proof}

\subsection{Proof of Theorem \ref{xiaoyu}}
The proof of Theorem \ref{xiaoyu} is similar to the proof Theorem \ref{dayu}.

\begin{proof}[Completion of the proof of Theorem \ref{xiaoyu}]
We first prove the limsup part. Let $t >0$ be a parameter. Being similar to the proofs of the inequalities (\ref{chebyshev})--(\ref{huajian}), we deduce that
\begin{equation*}
\limsup_{n \to \infty} \frac{1}{n}\lambda \left\{x \in \mathbb{I}: \frac{\log q_n(x)}{n} \leq b - \varepsilon\right\} \leq t(b - \varepsilon) + \mathrm{P}(1+t/2).
\end{equation*}
for any $t>0$. Therefore,
\begin{equation*}
\limsup_{n \to \infty} \frac{1}{n}\log \lambda \left\{x \in \mathbb{I}: \frac{\log q_n(x)}{n} \leq b - \varepsilon\right\} \leq \theta_2(\varepsilon)
\end{equation*}
with
\[
\theta_2(\varepsilon) = \inf_{t> 0}\big\{t(b - \varepsilon) + \mathrm{P}(1+t/2)\big\}.
\]
Now we show that $\theta_2(\varepsilon) < 0$. For any $u >-1$, we consider the function
\[
h(u)= u(b - \varepsilon) + \mathrm{P}(1+u/2).
\]
Notice that $h(0) =0$ and $h^{\prime}(0) <0$ by (\ref{ele}), then $h(t) < 0$ for $t$ sufficiently close to $0$ and hence that $\theta_2(\varepsilon) < 0$.

Next we give the proof of the liminf part. It follows from the recursive formula (\ref{recursive}) that
\[
q_n(x)= a_n(x) q_{n-1}(x) + q_{n-2}(x) \leq (a_n(x)+1) q_{n-1}(x)\leq \cdots \leq (a_n(x)+1)\cdots(a_1(x)+1).
\]
So,
\begin{equation}\label{xiaoyu inf}
\left\{x \in \mathbb{I}: \frac{\log q_n(x)}{n} \leq b - \varepsilon\right\} \supseteq
\left\{x \in \mathbb{I}:\frac{\sum_{k=1}^n \log (a_k(x)+1)}{n} \leq b - \varepsilon\right\}.
\end{equation}
Let $b^\ast_\varepsilon$ be the largest integer less than or equal to $(e^{b -\varepsilon} -1)$. Here we remark that the assumption $0<\varepsilon \leq b- \log 2$ in Theorem \ref{xiaoyu} is just to guarantee that $b^\ast_\varepsilon \geq 1$.
Notice that the right-hand set in (\ref{xiaoyu inf}) contains the following set
\begin{align*}
 \Big\{x \in \mathbb{I}: \log (a_1(x)+1) \leq b - \varepsilon , \cdots, \log (a_n(x)+1) \leq b - \varepsilon\Big\}
\end{align*}
and this subset also contains the following cylinder
\begin{align*}
\Big\{x \in \mathbb{I}:a_1(x) = b^\ast_\varepsilon, \cdots, a_n(x)= b^\ast _\varepsilon\Big\},
\end{align*}
combing this with (\ref{xiaoyu inf}), we obtain that
\begin{align*}
\lambda \left\{x \in \mathbb{I}: \frac{\log q_n(x)}{n} \leq b - \varepsilon\right\} \geq
|I(\underbrace{b^\ast_\varepsilon, \cdots,b^\ast_\varepsilon}_{n})| \geq \left(\frac{1}{3b_\varepsilon^{\ast 2}}\right)^n,
\end{align*}
where the last inequality is from Lemma \ref{Khin57}. Therefore,
\[
\liminf_{n \to \infty} \frac{1}{n}\log \lambda \left\{x \in \mathbb{I}:\frac{\log q_n(x)}{n} \leq b - \varepsilon\right\} \geq - 2\log b^\ast_\varepsilon -\log 3.
\]
\end{proof}

\subsection{Proof of Theorem \ref{Large deviations}}

\begin{proof}[Completion of the proof of Theorem \ref{Large deviations}]
For any $\varepsilon>0$ and $n \geq 1$, since
\begin{align}\label{fenkai}
 &\lambda \left\{x \in \mathbb{I}: \left|\frac{\log q_n(x)}{n} - b\right| \geq \varepsilon\right\} \notag \\
=&\lambda \left\{x \in \mathbb{I}: \frac{\log q_n(x)}{n} \geq b + \varepsilon\right\} + \mathrm{P} \left\{x \in \mathbb{I}: \frac{\log q_n(x)}{n} \leq b - \varepsilon\right\},
\end{align}
we obtain that
\begin{align*}
 &\limsup_{n \to \infty} \frac{1}{n}\log \lambda \left\{x \in \mathbb{I}: \left|\frac{\log q_n(x)}{n} - b\right| \geq \varepsilon\right\} \\
=&\limsup_{n \to \infty} \frac{1}{n}\log \left(\lambda \left\{x \in \mathbb{I}: \frac{\log q_n(x)}{n} \geq b + \varepsilon\right\} + \lambda \left\{x \in \mathbb{I}: \frac{\log q_n(x)}{n} \leq b - \varepsilon\right\}\right)\\
\leq & \max\{\theta_1(\varepsilon),\theta_2(\varepsilon)\}
\end{align*}
where the last inequality follows from the limsups in Theorems \ref{dayu} and \ref{xiaoyu}. Therefore, for any $\varepsilon>0$, there exist positive real $\alpha$ (only depending on $\varepsilon$) and positive integer $N:=N_\varepsilon$ such that for all $n > N$, we have
\begin{equation}\label{probability}
\lambda \left\{x \in \mathbb{I}: \left|\frac{\log q_n(x)}{n} - b\right| \geq \varepsilon\right\} \leq e^{-\alpha n}.
\end{equation}
For any $1 \leq n \leq N$, since the probabilities of the left-hand side in (\ref{probability}) are bounded, we can choose sufficiently large $A$ (only depending on $\varepsilon$) such that
\[
\lambda \left\{x \in \mathbb{I}: \left|\frac{\log q_n(x)}{n} - b\right| \geq \varepsilon\right\} \leq Ae^{-\alpha n}
\]
holds for all $n \geq 1$. Thus, the upper bound of the result in Theorem \ref{Large deviations} is established. The lower bound of the result in Theorem \ref{Large deviations} can also be obtained using the similar methods.
\end{proof}

\section{Applications}
In this section, we will apply our results to the following quantities related to the denominator of convergent $q_n$ in continued fractions. The following notations $\theta_1$, $\theta_2$, $b_\varepsilon$ and $b^\ast_\varepsilon$ are as defined in the Theorems \ref{dayu} and \ref{xiaoyu}.

\subsection{Lyapunov exponents}
Lyapunov exponents measure the exponential rate of divergence of infinitesimally close orbits of a dynamical system. These exponents are intimately related with the global stochastic behavior of the system and are fundamental invariants of a dynamical system. Here we define the \emph{Lyapunov exponent} of the continued fraction transformation $T$ by
\[
l(x):= \lim\limits_{n \to \infty}\frac{1}{n}\log |(T^n)^{\prime}(x)|
\]
if the limit exists. It is well known (see \cite{lesPW99}) that there exists a positive constant $K$ such that for any $x \in [0,1)$,
\[
\frac{1}{2K}q^2_n(x) \leq |(T^n)^{\prime}(x)| \leq Kq^2_n(x).
\]
By Theorem \ref{Levy theorem} and this result, we have that $l(x)$ is constantly $\pi^2/(6\log2)$ for $\lambda$-almost all $x \in [0,1)$.
Combing this with Theorems \ref{dayu}, \ref{xiaoyu} and \ref{Large deviations}, we know

\begin{theorem}
For any $\varepsilon >0$,
\begin{equation*}
\limsup_{n \to \infty} \frac{1}{n}\log \lambda\left\{x\in [0,1):\frac{1}{n}\log |(T^n)^{\prime}(x)| - 2b \geq \varepsilon\right\} \leq \theta_1(\varepsilon/2)
\end{equation*}
 and
\begin{equation*}
\liminf_{n \to \infty} \frac{1}{n}\log \lambda\left\{x\in [0,1):\frac{1}{n}\log |(T^n)^{\prime}(x)| - 2b \geq \varepsilon\right\} \geq -2\log b_{\varepsilon/2}-\log 3.
\end{equation*}
\end{theorem}

\begin{theorem}
For any $0< \varepsilon \leq 2b$,
\begin{equation*}
\limsup_{n \to \infty} \frac{1}{n}\log\lambda\left\{x\in [0,1):\frac{1}{n}\log |(T^n)^{\prime}(x)| - 2b \leq -\varepsilon\right\} \leq \theta_2(\varepsilon/2)
\end{equation*}
and for any $0< \varepsilon \leq 2(b -\log2)$,
\begin{equation*}
\liminf_{n \to \infty} \frac{1}{n}\log \lambda\left\{x\in [0,1):\frac{1}{n}\log |(T^n)^{\prime}(x)| - 2b \leq -\varepsilon\right\} \geq - 2\log b^\ast_{\varepsilon/2}- \log 3.
\end{equation*}
\end{theorem}


\begin{theorem}
For any $\varepsilon > 0$, there exist the constants $A_1, B_1> 0$ and $\alpha_1, \beta_1 > 0$ (both only depending on $\varepsilon$)such that for all $n \geq 1$, we have
\begin{equation*}
B_1e^{-\beta_1 n} \leq \lambda \left\{x \in [0,1): \left|\frac{1}{n}\log |(T^n)^{\prime}(x)| - \frac{\pi^2}{6\log2}\right| \geq \varepsilon\right\} \leq A_1e^{-\alpha_1 n}.
\end{equation*}
\end{theorem}

\subsection{The growth rate of Diophantine approximation}
For any $x \in [0,1)$ with its continued fraction expansion (\ref{continued fraction expansion}), we define
\[
d(x): = \lim\limits_{n \to \infty}\frac{1}{n}\log \left|x - \frac{p_n(x)}{q_n(x)}\right|
\]
if the limit exists. It is clear to see that this function stands for the rate of rational numbers approximating to real numbers. By Theorem \ref{Levy theorem} and Diophantine inequalities (\ref{diophantine}), we know that the quantity $d(x)= -\pi^2/(6\log2)$ for $\lambda$-almost all $x \in [0,1)$. In view of (\ref{diophantine}), we obtain that
\begin{theorem}
For any $0< \varepsilon\leq 2b $,
\begin{equation*}
\limsup_{n \to \infty} \frac{1}{n}\log \lambda\left\{x\in [0,1):\frac{1}{n}\log \left|x - \frac{p_n(x)}{q_n(x)}\right| +2b \geq \varepsilon\right\} \leq \theta_2(\varepsilon/2)
\end{equation*}
 and for any $0< \varepsilon\leq 2(b-\log 2)$,
\begin{equation*}
\liminf_{n \to \infty} \frac{1}{n}\log \lambda\left\{x\in [0,1):\frac{1}{n}\log \left|x - \frac{p_n(x)}{q_n(x)}\right| + 2b \geq \varepsilon\right\} \geq -2\log b^\ast_{\varepsilon/2}-\log 3.
\end{equation*}
\end{theorem}

\begin{theorem}
For any $\varepsilon>0$,
\begin{equation*}
\limsup_{n \to \infty} \frac{1}{n}\log \lambda\left\{x\in [0,1):\frac{1}{n}\log \left|x - \frac{p_n(x)}{q_n(x)}\right| + 2b \leq -\varepsilon\right\} \leq \theta_1(\varepsilon/2)
\end{equation*}
and
\begin{equation*}
\liminf_{n \to \infty} \frac{1}{n}\log \lambda\left\{x\in [0,1):\frac{1}{n}\log \left|x - \frac{p_n(x)}{q_n(x)}\right| + 2b \leq -\varepsilon\right\} \geq -2\log b_{\varepsilon/2}-\log 3.
\end{equation*}
\end{theorem}

\begin{theorem}\label{Diophantine}
For any $\varepsilon > 0$, there exist the constants $A_2, B_2> 0$ and $\alpha_2, \beta_2 > 0$ (both only depending on $\varepsilon$) such that for all $n \geq 1$, we have
\begin{equation*}
B_2e^{-\beta_2 n} \leq \lambda \left\{x \in [0,1):\left|\frac{1}{n}\log\left|x - \frac{p_n(x)}{q_n(x)}\right| + \frac{\pi^2}{6\log2}\right| \geq \varepsilon\right\} \leq A_2e^{-\alpha_2 n}.
\end{equation*}
\end{theorem}

\subsection{The growth rate of the length of cylinders}
In dynamical system, the theorem of Shannon-Mcmillan-Breiman (see \cite[Theorem 6.2.1]{lesDK02}) states that for every generating partition on an ergodic system of finite entropy, the exponential decay rate of the measure of cylinder sets equals the metric entropy almost everywhere. Now we consider the continued fractions dynamical system $([0,1), \mathcal{B}, T, \mu)$, where $\mathcal{B}$ is the Borel $\sigma$-algebra on $[0,1)$ and $\mu$ is the \emph{Gauss measure} with a bounded density $\frac{1}{(1+x) \log 2}$ on $[0,1)$ with respect to Lebesgue measure. For any $x \in [0,1)$, we put
\[
s(x): = \lim\limits_{n \to \infty}\frac{1}{n}\log \mu(I_n(x))
\]
if the limit exists, where $I_n(x)$ denotes the $n$-th order cylinder containing $x$.
It is clear to see that $s(x)$ is alternatively defined by $s(x) = \lim_{n \to \infty} (\log |I_n(x)|)/n$ because of the relation between Gauss measure and Lebesgue measure. Shannon-Mcmillan-Breiman's theorem yields that $s(x)$ exits and is equal to $-\pi^2/(6\log2)$ for $\lambda$-almost all $x \in [0,1)$.
Notice that
\[
\frac{1}{2q^2_n(x)} \leq |I_n(x)| = \frac{1}{q_n(x)(q_n(x)+q_{n-1}(x))} \leq \frac{1}{q^2_n(x)}
\]
(see \cite{lesDK02, lesIK02}), in view of Theorems \ref{dayu}, \ref{xiaoyu} and \ref{Large deviations}, we have

\begin{theorem}
For any $0<\varepsilon \leq 2b$,
\begin{equation*}
\limsup_{n \to \infty} \frac{1}{n}\log \lambda\left\{x\in [0,1):\frac{1}{n}\log |I_n(x)| + 2b \geq \varepsilon\right\} \leq \theta_2(\varepsilon/2)
\end{equation*}
 and for any $0<\varepsilon \leq 2(b -\log 2)$,
\begin{equation*}
\liminf_{n \to \infty} \frac{1}{n}\log \lambda\left\{x\in [0,1):\frac{1}{n}\log |I_n(x)| + 2b \geq \varepsilon\right\} \geq -2\log b^\ast_{\varepsilon/2}-\log 3.
\end{equation*}
\end{theorem}

\begin{theorem}
For any $\varepsilon >0$,
\begin{equation*}
\limsup_{n \to \infty} \frac{1}{n}\log \lambda\left\{x\in [0,1):\frac{1}{n}\log |I_n(x)| + 2b \leq - \varepsilon\right\} \leq \theta_1(\varepsilon/2)
\end{equation*}
 and
\begin{equation*}
\liminf_{n \to \infty} \frac{1}{n}\log \lambda\left\{x\in [0,1):\frac{1}{n}\log |I_n(x)| + 2b \leq -\varepsilon\right\} \geq -2\log b_{\varepsilon/2}-\log 3.
\end{equation*}
\end{theorem}

\begin{theorem}\label{Diophantine}
For any $\varepsilon > 0$, there exist the constants $A_3, B_3> 0$ and $\alpha_3, \beta_3 > 0$ (both only depending on $\varepsilon$) such that for all $n \geq 1$, we have
\begin{equation*}
B_3e^{-\beta_3 n} \leq \lambda \left\{x \in [0,1):\left|\frac{1}{n}\log |I_n(x)| + \frac{\pi^2}{6\log2}\right| \geq \varepsilon\right\} \leq A_3e^{-\alpha_3 n}.
\end{equation*}
\end{theorem}

{\bf Acknowledgement}
The work was supported by NSFC 11371148, Guangdong Natural Science Foundation 2014A030313230, and "Fundamental Research Funds for the Central Universities" SCUT 2015ZZ055 and 2015ZZ127.


\begin{thebibliography}{10}
\bibitem{lesBS00} L. Barreira and J. Schmeling, {\it Sets of ``non-typical" points have full topological entropy and full Hausdorff dimension}, Israel J. Math. 116 (2000), 29--70.


\bibitem{lesBaxa99} C. Baxa, {\it On the growth of the denominators of convergents}, Acta Math. Hungar. 83 (1999), no. 1--2, 125--130.


\bibitem{lesBaxa09} C. Baxa, {\it L\'{e}vy constants of transcendental numbers}, Proc. Amer. Math. Soc. 137 (2009), no. 7, 2243--2249.


\bibitem{lesDK02} K. Dajani and C. Kraaikamp, {\it Ergodic Theory of Numbers}, Math. Assoc. America, Washington, DC, 2002.


\bibitem{lesFai92} C. Faivre, {\it Distribution of L\'{e}vy constants for quadratic numbers}, Acta Arith. 61 (1992), no. 1, 13--34.


\bibitem{lesFai97} C. Faivre, {\it The L\'{e}vy constant of an irrational number}, Acta Math. Hungar. 74 (1997), no. 1--2, 57--61.


\bibitem{lesF.L.W.W} A.-H. Fan, L.-M. Liao, B.-W. Wang and J. Wu, {\it On Khintchine exponents and Lyapunov exponents of continued fractions}, Ergodic Theory Dynam. Systems 29 (2009), no. 1, 73--109.


\bibitem{lesFWSL15} L. Fang, M. Wu, N.-R. Shieh and B. Li, {\it Random continued fractions: L\'{e}vy constant and Chernoff-type estimate}, J. Math. Anal. Appl. 429 (2015), no. 1, 513--531.





\bibitem{lesGR70} M. Gordin and M. Reznik, {\it The law of the iterated logarithm for the denominators of continued fractions}, Vestnik Leningrad. Univ. 25 (1970), 28--33.








\bibitem{lesIbr61} I. Ibragimov, {\it A theorem from the metric theory of continued fractions}, Vestnik Leningrad. Univ. 16 (1961), no. 1, 13--24.

\bibitem{lesIK02} M. Iosifescu and C. Kraaikamp, {\it Metrical Theory of Continued Fractions. Mathematics and Its Applications}, Kluwer Academic Publishers, Dordrecht, 2002.


\bibitem{lesKS07} M. Kesseb\"{o}hmer and B. Stratmann, {\it A multifractal analysis for Stern-Brocot intervals, continued fractions and Diophantine growth rates}, J. Reine Angew. Math. 605 (2007), 133--163.


\bibitem{lesKhi64} Y. Khintchine, {\it Continued Fractions}, The University of Chicago Press, Chicago, 1964.

\bibitem{lesLevy29} P. L\'{e}vy, {\it Sur les lois de probabilit\'{e} dont d\'{e}pendent les quotients complets et incomplets d\'{u}ne fraction continue}, Bull. Soc. Math. France 57 (1929), 178--194.


\bibitem{lesMay90} D. Mayer, {\it On the thermodynamic formalism for the Gauss map}, Comm. Math. Phys. 130 (1990), no. 2, 311--333.


\bibitem{lesMor94} T. Morita, {\it Local limit theorem and distribution of periodic orbits of Lasota-Yorke transformations with infinite Markov partition}, J. Math. Soc. Japan. 46 (1994), no. 2, 309--343.


\bibitem{lesPS69} W. Philipp and O. Stackelberg, {\it Zwei Grenzwerts\"{a}tze f\"{u}r Kettenbr\"{u}che}, Math. Ann. 181 (1969), 152--154.


\bibitem{lesPW99} M. Pollicott and H. Weiss, {\it Multifractal analysis of Lyapunov exponent for continued fraction and Manneville-Pomeau transformations and applications to Diophantine approximation}, Comm. Math. Phys. 207 (1999), no. 1, 145--171.



\bibitem{lesWu06a} J. Wu, {\it A remark on the growth of the denominators of convergents}, Monatsh. Math. 147 (2006), no. 3, 259--264.


\bibitem{lesWu06} J. Wu, {\it On the L\'{e}vy constants for quadratic irrationals}, Proc. Amer. Math. Soc. 134 (2006), no. 6, 1631--1634.
\end{thebibliography}
\end{document}